\documentclass[conference]{IEEEtran}
\IEEEoverridecommandlockouts
\usepackage{cite}
\usepackage{amsmath,amssymb,amsfonts}
\usepackage{algorithmic}
\usepackage{graphicx}
\usepackage{textcomp}
\usepackage{xcolor}
\def\BibTeX{{\rm B\kern-.05em{\sc i\kern-.025em b}\kern-.08em
    T\kern-.1667em\lower.7ex\hbox{E}\kern-.125emX}}

\usepackage{graphics} 
\usepackage{graphicx}
\usepackage{subfigure}
\usepackage{amssymb}  
\usepackage{cite}
\usepackage{amsmath}
\usepackage[boxed,ruled,lined]{algorithm2e}
\usepackage{epstopdf}
\usepackage{booktabs}
\usepackage{multirow}
\usepackage{etoolbox}
\usepackage{mathdots}
\usepackage{xcolor, float}
\makeatletter
\patchcmd{\@makecaption}
{\scshape}
{}
{}
{}
\makeatother
\usepackage{diagbox}
\usepackage{bm}
\usepackage{subfigure}
\usepackage{lipsum}
\usepackage{braket,amsfonts,amsopn}
\usepackage{hyperref}
\usepackage{amsthm} 

\def\twon #1{\left\|#1\right\|_2}

\def\frobn #1{\left\|#1\right\|_{\text{F}}}

\def\abs #1{\left|#1\right|}

\def\st{\text{subject to }}

\def\bC{\mathbb{C}}

\def\bR{\mathbb{R}}

\def\bN{\mathbb{N}}
\def\bS{\mathbb{S}}

\def\bL{\mathbb{L}}

\def\m #1{\boldsymbol{#1}}

\def\cD{\mathcal{D}}

\def\cH{\mathcal{H}}
\def\cI{\mathcal{I}}

\def\cO{\mathcal{O}}
\def\cP{\mathcal{P}}

\def\cT{\mathcal{T}}

\def\cW{\mathcal{W}}

\def\cG{\mathcal{G}}

\def\bee{\begin{equation}}
	\def\ene{\end{equation}}

\def\beq{\begin{eqnarray}}
	\def\enq{\end{eqnarray}}

\newtheorem{thm}{Theorem}

\def\equ #1{\begin{equation}#1\end{equation}}

\def\sbra #1{\left(#1\right)}
\def\mbra #1{\left[#1\right]}
\def\lbra #1{\left\{#1\right\}}
\def\diag #1{\text{diag}#1}

\def\rank #1{\text{rank}#1}
\def\st {\text{ subject to }}

\allowdisplaybreaks[4]

\begin{document}

\title{Low-Complexity Algorithms for Multichannel Spectral Super-Resolution \\
\thanks{The research of the project was supported by the National Natural Science Foundation of China under Grant 12371464.}
}

\author{Xunmeng Wu, Zai Yang, and Zongben Xu \\
	\textit{School of Mathematics and Statistics, Xi’an Jiaotong University, China} \\
	\textit{wxm1996@stu.xjtu.edu.cn; yangzai@xjtu.edu.cn; zbxu@xjtu.edu.cn}
}

\maketitle

\begin{abstract}
This paper studies the problem of multichannel spectral super-resolution with either constant amplitude (CA) or not. We propose two optimization problems based on low-rank Hankel-Toeplitz matrix factorization. The two problems effectively leverage the multichannel and CA structures, while also enabling the design of low-complexity gradient descent algorithms for their solutions. Extensive simulations show the superior performance of the proposed algorithms.
\end{abstract}


\section{Introduction}
Spectral super-resolution refers to the process of estimating the frequencies of several complex sinusoids from their superimposed samples \cite{candes2014towards}. In this paper, we focus on the multichannel spectral super-resolution problem, also known as multichannel frequency estimation \cite{yang2018sample,wu2023multichannel,wu2024multichannel,li2022stability,liu2024mathematical,fei2023iff}, where the sinusoids among multiple channels share the same frequency. In particular, the full signal matrix $\m{X}^\star$ is composed of equispaced samples $\set{x_{jl}^\star}$, which are given by
\equ{
	x_{jl}^\star = \sum^K_{k=1} b_{kl}^\star e^{i \sbra{-2\pi f_k^\star \cdot (j-1) + \phi_{kl}^\star } }, \; j \in \bN, \; l \in \bL,
}
where $\bN \triangleq \lbra{ 1, \ldots, N }$, $\bL \triangleq \lbra{ 1, \ldots, L }$, $N$ is the full sample size for each channel, $L$ and $K$ are the number of channels and frequencies, respectively, $f_k^\star \in \mathbb{F} \triangleq [0,1)$, $b_{kl}^\star > 0$, and $\phi_{kl}^\star \in \bR$ denote the $k$th unknown normalized frequency, amplitude, and phase in the $l$-th channel, respectively, and $i=\sqrt{-1}$. When $b_{kl}^\star = b_k^\star$ for $l \in \bL$, $\m{X}^\star$ is known as a constant amplitude (CA) signal, also referred to as a constant modulus signal \cite{van1996analytical,leshem1999direction,wu2024direction,wu2024multichannel}; otherwise, it is classified as a general multichannel signal.

We consider a compressive setting in which only a subset $\Omega$ of rows of $\m{X}^\star$ is observed. Specifically, we have access to the incomplete signal matrix $\cP_{\Omega}(\m{X}^\star)$ that is given by
\equ{
	\cP_{\Omega}(\m{X}^\star) = \sum_{j\in \Omega} \m{e}_j^T \m{X}^\star , \label{eq:ob}
}
where $\Omega \subset \bN$ with $\abs{\Omega} = M$ and $\m{e}_j$ is a canonical basis of $\bR^N$. Our goal is to estimate $\lbra{f_k^\star}$ from $\cP_{\Omega}(\m{X}^\star)$.

Multichannel spectral super-resolution is an essential topic within statistical signal processing and has been researched under various related areas such as harmonic retrieval in spectral analysis \cite{stoica2005spectral} and direction-of-arrival estimation in array processing \cite{krim1996two,yang2018sparse,pesavento2023three,kalra2024stable,qiao2019guaranteed}. The classical subspace methods, including MUSIC \cite{schmidt1986multiple} and ESPRIT \cite{roy1989esprit}, can achieve arbitrarily high resolution from complete noiseless samples; however, their performance deteriorates in the case of incomplete observations concerned in this paper. To overcome the limitation of subspace methods, sparse optimization and compressed sensing methods \cite{yang2016exact,yang2018sample,li2015off,steffens2018compact,cai2018spectral,wu2024multichannel,li2024effectiveness,wang2018ivdst,castanheira2019low,hansen2018superfast,mao2022blind,yao2024accelerated}, particularly the atomic norm minimization (ANM) approaches \cite{yang2016exact,yang2018sample,li2015off,wu2024multichannel,li2024effectiveness}, have been proposed. Leveraging the fact that $K \ll N$, ANM formulates super-resolution as a spectral sparse signal recovery problem, which can be cast as a semidefinite program (SDP). This optimization framework of ANM is flexible in dealing with the incomplete observations by treating them as variables to optimize. However, solving the SDP typically suffers from high computational complexity, which restricts its applicability in large-scale problems.

For single-channel spectral super-resolution, a fast and effective gradient descent algorithm for the Hankel-Toeplitz matrix factorization problem has been proposed in \cite{wu2023spectral}. It remains unclear how to extend this approach to the multichannel scenarios addressed in this paper. Notably, two rank-constrained positive semi-definite (PSD) Hankel-Toeplitz matrix optimization problems have been proposed for general multichannel signals in \cite{wu2023multichannel} and CA signals in \cite{wu2024direction}, respectively. Although both methods exhibit good accuracy, they require truncated eigen-decomposition in each iteration to handle the rank and PSD constraints, resulting in a rather high computational complexity. This prompts us to develop associated low-complexity algorithms.

In this paper, for multichannel spectral super-resolution with and without CA, we propose two corresponding optimization problems along with low-complexity algorithms based on low-rank Hankel-Toeplitz matrix factorization. Specifically, we factorize the rank-constrained PSD Hankel-Toeplitz matrix for each channel, proposed in \cite{wu2023multichannel,wu2024direction}, as a product of a low-rank factor matrix and its conjugate transpose, thereby eliminating the rank and PSD constraints. We then impose certain linear constraints on the factor matrices to effectively utilize the multichannel and CA structures. Consequently, we formulate two unconstrained optimization problems: one for general multichannel signals and another for CA signals. To solve these problems, we propose two low-complexity gradient descent algorithms, where the gradients can be efficiently computed via the fast Fourier transform (FFT).

\emph{Notations:} For vector $\m{x}$, its complex conjugate and $\ell_2$-norm are denoted by $\overline{\m{x}}$ and $\twon{\m{x}}$, respectively.
For matrix $\m{X}$, its transpose, complex conjugate, conjugate transpose, and Frobenius norm are denoted by $\m{X}^T$, $\overline{\m{X}}$, $\m{X}^H$, and $\frobn{\m{X}}$, respectively. The $(i,j)$-th entry of $\m{X}$ is $x_{ij}$, and the $i$-th row and $j$-th column of $\m{X}$ are $\m{X}_{i,:}$ and $\m{X}_{:,j}$, respectively. Let $\bS_K^+ \triangleq \lbra{\m{X}: \m{X} \succeq \m{0}, \rank\sbra{\m{X}} \le K}$ denote the set of rank-constrained PSD matrices. We denote $\odot$ as the Hadamard product and $\cI$ as the identity operator. 

\section{Algorithm for General Multichannel Signals} \label{sec:2}
Since directly estimating the frequencies from the incomplete observation is difficult, it is natural to consider the following signal recovery problem:
\equ{
	\text{find } \m{X}, \st \cP_{\Omega}(\m{X}) = \cP_{\Omega}(\m{X}^\star) \text{ and } \m{X} \in S_0^L, \label{eq:rec_x}
}
where the set of general multichannel signals $S_0^L$ is given by
\equ{
	S_0^L \triangleq \lbra{ \m{A}\sbra{\m{f}}\m{S}: f_k \in \mathbb{F}, \m{S} \in \bC^{K \times L} }, \nonumber
} 
with $\m{A}\sbra{\m{f}} =\mbra{\m{a}(f_1),\ldots,\m{a}(f_K)}$ and $\mbra{\m{a}(f)}_j = e^{-i2\pi f(j-1)},j\in \bN$. 
Given the recovered signal, the frequency vector $\m{f}$ can be retrieved using subspace methods \cite{schmidt1986multiple,roy1989esprit}.

The problem in \eqref{eq:rec_x} can be equivalently transformed to the following weighted least squares problem:
\equ{
	\min_{\m{X}} \sum_{l=1}^L \twon{\cP_{\Omega} \cD (\m{X}_{:,l}-\m{X}^\star_{:,l})}^2, \st \m{X} \in S_0^L, \label{eq:p1}
} 
where $\cD$ is a weighted operator that is defined as  $\cD\m{x}=\m{\omega} \odot \m{x}$. The $N\times 1$ weight vector $\m{\omega}$ will be specified later. Without loss of generality, we assume that $N=2n-1$ for some integer $n$. If $N=2n-2$, we can consider the case where the last entry of a signal of length $N+1$ is unobserved. 

It is shown in \cite[Theorem 3]{wu2023multichannel} that when $K<n$, $S_0^L$ is equivalent to $S_{\text{HT}}^L$, which is given by
\equ{
	\begin{split}
		S_{\text{HT}}^L & \triangleq  \left\{ {\m{X}:\; \begin{bmatrix}\cT\overline{\m{t}^l} & \cH \overline{\m{X}}_{:,l} \\ \cH \m{X}_{:,l} & \cT \m{t} \end{bmatrix} \in \bS_K^+, \; l\in \bL,} \right.\\
		& \left.{  \text{for some } \lbra{\m{t}^l \in \bC^N },\m{t} \in \bC^N \text{ satisfying } \sum_{l=1}^L\m{t}^l = L\m{t}} \right\},  
	\end{split} \nonumber 
}
where the Hankel matrices $\cH\m{X}_{:,l} \in \bC^{n\times n}$ and the Toeplitz matrices $\cT\m{t}^l, \cT\m{t} \in \bC^{n\times n}$. Consequently, \eqref{eq:p1} is equivalent to the following Hankel-Toeplitz matrix optimization problem:
\equ{
	\min_{\m{X}} \sum_{l=1}^L \twon{\cP_{\Omega} \cD (\m{X}_{:,l}-\m{X}^\star_{:,l})}^2, \st \m{X} \in S_{\text{HT}}^L. \label{eq:p2}
} 

The ADMM algorithm can be applied to solve the problem in \eqref{eq:p2}. However, the rank and PSD constraints in $S_{\text{HT}}^L$ necessitates truncated eigen-decomposition in each iteration, leading to a rather high computational complexity of $\mathcal{O}(LKN^2)$ \cite{wu2023multichannel}. 

\subsection{Proposed Optimization Problem} \label{sec:Problem_MHTGD}
To develop a low-complexity algorithm, we reparameterize the low-rank PSD Hankel-Toeplitz matrices by the following factorization
\equ{
	\begin{bmatrix}\cT\overline{\m{t}^l} & \cH \overline{\m{X}}_{:,l} \\ \cH \m{X}_{:,l} & \cT \m{t} \end{bmatrix} = \begin{bmatrix} \m{Z}_1^l \\ \m{Z}_2^l \end{bmatrix} \begin{bmatrix} \m{Z}_1^l \\ \m{Z}_2^l \end{bmatrix}^H = \begin{bmatrix} \m{Z}_1^l \m{Z}_1^{l,H} & \m{Z}_1^l\m{Z}_2^{l,H} \\ \m{Z}_2^l \m{Z}_1^{l,H} & \m{Z}_2^l \m{Z}_2^{l,H} \end{bmatrix},
}
where $\m{Z}_1^l$ and $\m{Z}_2^l$ are $n\times K$ factor matrices. Subsequently, we propose a novel set $S_{\text{MF}}^L$ of multichannel Hankel-Toeplitz matrix factorization as follows:
\equ{
	\begin{split}
		S_{\text{MF}}^L =  & \left\{  {\m{X}:\, \cH\m{X}_{:,l} =\m{Z}_2^l \m{Z}_1^{l,H}, \cT\overline{\m{t}^l}=\m{Z}_1^l\m{Z}_1^{l,H},} \right.\\
		& \left. { \qquad \; \; \frac{1}{L} \sum_{q=1}^L \cT \m{t}^q = \m{Z}_2^l \m{Z}_2^{l,H}, \; l \in \bL,} \right.\\
		& \left. { \qquad \; \text{ for some } \m{t}^l \in\bC^N \; \text{and} \; \m{Z}_1^l, \m{Z}_2^l \in \bC^{n\times K},  l \in \bL } \right\} \! . 
	\end{split} \nonumber
}
\begin{thm} \label{thm:t1}
	Assume $K<n$. Then we have $S_{\text{MF}}^L = S_{\text{HT}}^L = S_0^L$.
\end{thm}

\begin{proof}
	According to \cite[Theorem 3]{wu2023multichannel}, we have $S_{\text{HT}}^L = S_0^L$. Then given any $\m{X} \in S_{\text{HT}}^L$, we have $\m{X} = \m{A}\sbra{\m{f}}\m{S}\in S_0^L$. Assume that $\twon{\m{S}_{k,:}} \neq 0, k = 1,\ldots,K$; otherwise, $\m{X}$ would be composed of $K'$ sinusoids where $K'<K$. In this case, we let $K=K'$. Let $\m{p}$ and $\m{p}^l$ be two $K\times 1$ vectors defined by $p_k = \frac{\twon{\m{S}_{k,:}}}{\sqrt{L}}$ and $p_k^l = \frac{\abs{s_{kl}}^2}{p_k}$, respectively, for $k = 1, \ldots, K$. It immediately follows that $\sum^L_{l=1} \m{p}^l = L \m{p}$. Let $\phi \sbra{ \m{S}_{:,l} }$ denote the phase vector of $\m{S}_{:,l}$, satisfying $\mbra{\phi \sbra{ \m{S}_{:,l} }}_k = s_{kl} / \abs{s_{kl}}$. Let $\m{A}$ denote the matrix formed by the first $n$ rows of $\m{A}\sbra{\m{f}}$. Let 
	\equ{
		\begin{split}
			& \m{Z}_1^l = \overline{\m{A}} \diag \sbra{ \m{p} }^{-\frac{1}{2}} \diag\sbra{\overline{\m{S}}_{:,l}} \diag \sbra{ \phi \sbra{ \m{S}_{:,l} } }, \\
			& \m{Z}_2^l = \m{A} \diag \sbra{\m{p}}^{\frac{1}{2}} \diag \sbra{ \phi \sbra{ \m{S}_{:,l} } }.
		\end{split} \label{eq:Z1Z2}
	}
	Subsequently, we can verify that
	\equ{
		\begin{split}
			& \m{Z}_2^l \m{Z}_1^{l,H} =  \m{A} \diag\sbra{\m{S}_{:,l}} \m{A}^T  = \cH\m{X}_{:,l}, \\
			& \m{Z}_1^l \m{Z}_1^{l,H} = \overline{\m{A}} \diag\sbra{ \m{p}^l } \m{A}^T = \cT \overline{\m{t}^l} \text{ for some } \m{t}^l, \\
			& \m{Z}_2^l \m{Z}_2^{l,H} = \m{A} \diag \sbra{\m{p}} \m{A}^H = \frac{1}{L} \m{A} \diag \sbra{\m{p}^l} \m{A}^H = \frac{1}{L} \sum_{q=1}^L \cT \m{t}^q.
		\end{split} \nonumber
	} 
	Consequently, we have $\m{X}\in S_{\text{MF}}^L$, and thus $S_{\text{HT}}^L \subseteq S_{\text{MF}}^L$. 
	
	On the other hand, for any $\m{X}\in S_{\text{MF}}^L$, we have 
	\equ{
		\begin{split}
			\begin{bmatrix}\cT\overline{\m{t}^l} & \cH \overline{\m{X}}_{:,l} \\ \cH \m{X}_{:,l} & \cT \m{t} \end{bmatrix} & = \begin{bmatrix}\cT\overline{\m{t}^l} & \cH \overline{\m{X}}_{:,l} \\ \cH \m{X}_{:,l} & \frac{1}{L} \sum_{q=1}^L \cT \m{t}^q \end{bmatrix} \\
			& = \begin{bmatrix} \m{Z}_1^l \m{Z}_1^{l,H} & \m{Z}_1^l\m{Z}_2^{l,H} \\ \m{Z}_2^l \m{Z}_1^{l,H} & \m{Z}_2^l \m{Z}_2^{l,H} \end{bmatrix} = \begin{bmatrix} \m{Z}_1^l \\ \m{Z}_2^l \end{bmatrix} \begin{bmatrix} \m{Z}_1^l \\ \m{Z}_2^l \end{bmatrix}^H.
		\end{split} \nonumber
	}
	This means that the Hankel-Toeplitz matrices belong to the set $\bS_K^+$. Consequently, we arrive at $\m{X}\in S_{\text{HT}}^L$ and $S_{\text{MF}}^L \subseteq S_{\text{HT}}^L$. Finally, we conclude that $S_{\text{MF}}^L = S_{\text{HT}}^L$.
\end{proof}

Using Theorem \ref{thm:t1}, we can rewrite the optimization problem in \eqref{eq:p2} as follows:
\equ{ 
	\begin{split}
		\min_{\m{X},\lbra{\m{t}^l},\lbra{\m{Z}_1^l,\m{Z}_2^l}} \sum_{l=1}^L  & \twon{\cP_{\Omega} \cD (\m{X}_{:,l}-\m{X}^\star_{:,l})}^2, \\
		 \st \cH\m{X}_{:,l} & =\m{Z}_2^l\m{Z}_1^{l,H}, \\
		\qquad \qquad \; \; \cT\overline{\m{t}^l} &=\m{Z}_1^l\m{Z}_1^{l,H}, \\
		\qquad \qquad \; \; \frac{1}{L} \sum_{q=1}^L \cT \m{t}^q & = \m{Z}_2^l \m{Z}_2^{l,H}, \; l\in \bL.
	\end{split} \label{eq:p3}
}

To reduce the number of variables to be optimized in \eqref{eq:p3}, we define the weight vector $\m{\omega}$ in $\cD$ such that $\omega_j = \sqrt{a_j}$, where $a_j$ represents the number of elements in the $j$th skew-diagonal of an $n \times n$ matrix.
It can be shown that $\cH^*\cH = \cD^2$ and $\cT^*\cT = \cD^2$. We then define $\cG \triangleq \mathcal{H}\mathcal{D}^{-1}$ and $\cW \triangleq \mathcal{T}\mathcal{D}^{-1}$.
Let $\cG^*$ and $\cW^*$ be the adjoint of $\cG$ and $\cW$, respectively.
It follows that $\cG^*\cG=\mathcal{I}$ and $\cW^*\cW=\mathcal{I}$. 
Subsequently, the constraints in \eqref{eq:p3} are satisfied if and only if the following conditions hold:
\equ{
	\begin{split}
		 (\mathcal{I}-\cG\cG^*)(\m{Z}_2^l\m{Z}_1^{l,H}) &=  \m{0}, \\ 
		 (\mathcal{I}-\cW\cW^*)(\m{Z}_1^l\m{Z}_1^{l,H}) & =  \m{0}, \\
		\sum_{q=1}^L \overline{\m{Z}_1^q} \m{Z}_1^{q,T} &= L \m{Z}_2^l \m{Z}_2^{l,H}, \; l \in \bL.
	\end{split} \label{eq:constraint}
}
Additionally, we introduce a new matrix $\m{Y}$ such that $\m{Y}_{:,l} = \cD \m{X}^\star_{:,l}$ for $l\in \bL$. Let us denote $\m{Z}^l = \mbra{ \m{Z}_1^{l,T}, \m{Z}_2^{l,T}}^T$. Then the problem in \eqref{eq:p3} is reformulated as follows:
\equ{
	\begin{split}
		\min_{\lbra{\m{Z}^l}} \sum^L_{l=1} \twon{ \cP_{\Omega} \cG^* \sbra{ \m{Z}_2^l\m{Z}_1^{l,H} - \cG\m{Y}_{:,l} } }^2,  \st \eqref{eq:constraint}.
	\end{split} \label{eq:p4}
}

Denote the sampling ratio $p = M/N$. Putting the objective function and the constraints in \eqref{eq:p4} together yields the following equivalent unconstrained minimization problem:
\begin{equation} 
	\begin{split}
		\min_{\lbra{\m{Z}^l}} f\sbra{\lbra{\m{Z}^l}} \triangleq 
		& \sum^L_{l=1}  \left[ { \frac{1}{2p} \twon{ \cP_{\Omega} \cG^* \sbra{ \m{Z}_2^l\m{Z}_1^{l,H} - \cG\m{Y}_{:,l} } }^2 } \right.\\
		& \left. { \qquad + \frac{1}{2}\frobn{(\mathcal{I}-\cG\cG^*)(\m{Z}_2^l\m{Z}_1^{l,H})}^2  } \right.\\
		& \left. { \qquad + \frac{1}{4}\frobn{(\mathcal{I}-\cW\cW^*)(\m{Z}_1^l \m{Z}_1^{l,H})}^2 } \right.\\
		& \left. { \qquad + \frac{1}{4} \frobn{ \sum_{q=1}^L \overline{\m{Z}_1^q} \m{Z}_1^{q,T} - L \m{Z}_2^l \m{Z}_2^{l,H} }^2 } \right].
	\end{split} \label{eq:p_f}
\end{equation} 

\subsection{Proposed Algorithm}
We propose an efficient gradient descent algorithm, designated as MHTGD, to solve the problem in \eqref{eq:p_f}; see Algorithm \ref{alg:MHTGD}.
\begin{algorithm}
	\caption{Multichannel Hankel-Toeplitz matrix factorization-based Gradient Descent (MHTGD)}
	\begin{algorithmic}
		\REQUIRE $\Gamma_K\sbra{p^{-1}\cG\mathcal{P}_{\Omega}(\m{Y}_{:,l})} = \m{U}^{l}\m{\Sigma}^l \m{V}^{l,H}$, $\m{Z}_1^{l,0} = \m{V}^l \m{\Sigma}^{l,1/2}$, and $\m{Z}_2^{l,0} = \m{U}^l \m{\Sigma}^{l,1/2}$ for $l\in \bL$.
		\WHILE{$t<t_{\text{max}}$ or the stopping criterion is not met}
		\STATE $\m{Z}^{l,t+1} = \m{Z}^{l,t} -\eta \nabla f(\m{Z}^{l,t}), \; l \in \bL.$
		\STATE $t=t+1$.
		\ENDWHILE
		\ENSURE $\set{\widehat{\m{Z}}^l}$, $\widehat{\m{X}}_{:,l} = \mathcal{D}^{-1}\cG^*\sbra{\widehat{\m{Z}}_2^l \widehat{\m{Z}}_1^{l,H}},\; l \in \bL.$
	\end{algorithmic} 	\label{alg:MHTGD}
\end{algorithm}
In the MHTGD algorithm, we first apply $K$-truncated singular value decomposition to obtain a rank-$K$ approximation of the matrix $p^{-1}\cG\mathcal{P}_{\Omega}(\m{Y}_{:,l})$, followed by the initialization of the factor matrices $\lbra{\m{Z}^l}$. Then each factor matrix $\m{Z}^l$ is updated iteratively in the direction of the negative gradient with a step size $\eta$.

The gradient $\nabla f(\m{Z}^l) = \mbra{ \nabla f(\m{Z}^1_1)^T, \nabla f(\m{Z}^L_2)^T }^T$, calculated by Wirtinger calculus \cite{hunger2007introduction}, is given by
\equ{ { \small 
	\begin{split}
		\nabla f(\m{Z}_1^l) 
		& = \cG\mbra{\frac{1}{p} \cP_{\Omega} \sbra{ \cG^* \sbra{ \m{Z}_1^l\m{Z}_2^{l,H}} - \overline{\m{Y}}_{:,l} } - \cG^* \sbra{\m{Z}_1^l\m{Z}_2^{l,H}} } \m{Z}_2^l \\
		& \quad -\cW\cW^* \sbra{\m{Z}_1^l \m{Z}_1^{l,H}} \m{Z}_1^l  + \m{Z}_1^l \sbra{ \m{Z}_1^{l,H} \m{Z}_1^l + \m{Z}_2^{l,H} \m{Z}_2^l } \\
		& \quad + L \sum_{q=1}^L \mbra{ \m{Z}_1^q \sbra{\m{Z}_1^{q,H} \m{Z}_1^l} - \overline{\m{Z}_2^q} \sbra{ \m{Z}_2^{q,T} \m{Z}_1^l} }, \; l \in \bL, \nonumber
	\end{split} 
	}
}
\equ{ { \small 
	\begin{split}
		\nabla f(\m{Z}_2^l) 
		& = \cG\mbra{\frac{1}{p} \cP_{\Omega} \sbra{ \cG^* \sbra{ \m{Z}_2^l\m{Z}_1^{l,H}} - \m{Y}_{:,l} } - \cG^* \sbra{\m{Z}_2^l\m{Z}_1^{l,H}} } \m{Z}_1^l \\
		& \quad + \m{Z}_2^l \sbra{ \m{Z}_1^{l,H} \m{Z}_1^l + L^2 \m{Z}_2^{l,H} \m{Z}_2^l } \\
		& \quad - L \sum^L_{q=1} \overline{\m{Z}_1^q} \sbra{\m{Z}_1^{q,T} \m{Z}_2^l}, \; l \in \bL.
	\end{split} \nonumber 
	}
}

In $\nabla f(\m{Z}^l)$, the operations $(\cG\m{v})\m{Z}$, $(\cW\m{v})\m{Z}$, $\cG^*\sbra{\m{Z}\m{Z}^H}$, and $\cW^*\sbra{\m{Z}\m{Z}^H}$, where $\m{v}\in\bC^N$ and $\m{Z}\in \bC^{n\times K}$, can be efficiently computed through fast convolutions using FFT, which costs $\cO(K N\log N)$ flops. The remaining matrix multiplications, such as $\m{Z}\sbra{\m{Z}^H\m{Z}}$, require $\cO (K^2 N)$ flops. Consequently, the MHTGD algorithm exhibits a per-iteration computational complexity of $\cO\sbra{LKN\log N + L^2K^2N}$. Since the gradient computations for each channel $l$ are independent, we can leverage parallel computing to reduce the complexity to $\cO\sbra{KN\log N + K^2N}$. The complexity of MHTGD is more efficient than the $\cO\sbra{ L K N^2 }$ complexity of StruMER \cite{wu2023multichannel} and the $\cO\sbra{ (N+L^2)^{2} (N+L)^{2.5} }$ complexity of ANM \cite{yang2016exact,yang2018sparse}.

\section{Algorithm for Constant Amplitude Signals}
As in Section \ref{sec:2}, we consider the problem in \eqref{eq:p1}, where $S_0^L$ is replaced by the set of CA signals $S_0^C$, given by
\equ{
	S_0^C \triangleq \lbra{ \m{A}\sbra{\m{f}} \diag\sbra{\m{b}} \m{\Phi}: f_k \in \mathbb{F}, b_k > 0, \Phi_{kl} = e^{i\phi_{kl}} }. \nonumber
} 

It is shown in \cite{wu2024direction} that, when $K<n$, $S_0^C$ can be tightly relaxed to $S_{\text{HT}}^C$ that is given by
\equ{
	\begin{split}
		S_{\text{HT}}^C \triangleq \lbra{ \m{X}:\,\begin{bmatrix} \cT\overline{\m{t}} & \cH\overline{\m{X}}_{:,l}\\ \cH\m{X}_{:,l} & \cT\m{t} \end{bmatrix} \in \bS_K^+, l \in \bL, \text{for some } \m{t} }
	\end{split}. \nonumber
} 
To efficiently solve the resulting problem, we propose a novel set $S_{\text{MF}}^C$ as follows:
\equ{
	\begin{split}
		& S_{\text{MF}}^C \triangleq \left\{ {\m{X}:\,  \cH\m{X}_{:,l} =\m{Z}^l\m{Z}^{l,T}, \cT\m{t}=\m{Z}^l\m{Z}^{l,H}, \; l \in \bL, } \right.\\
		& \left. { \qquad \qquad \quad \text{ for some } \m{t}\in\bC^N \; \text{and} \; \m{Z}^l \in \bC^{n\times K}, \; l \in \bL}  \right\}.
	\end{split} \nonumber
}
Then, we have the following theorem, of which the proof is similar to that of Theorem \ref{thm:t1} and will be omitted.
\begin{thm} \label{thm:t2}
	Assume $K<n$. Then  $S_{\text{MF}}^C = S_{\text{HT}}^C$.
\end{thm}

Using Theorem \ref{thm:t2} and reformulating the resulting problem, as in Section \ref{sec:Problem_MHTGD}, we obtain the following one:
\equ{ 
	\begin{split}
		\min_{\lbra{\m{Z}^l}} g\sbra{\lbra{\m{Z}^l}}  \triangleq 
		& \sum^L_{l=1} \mbra{ {\frac{1}{4p} \twon{ \cP_{\Omega} \cG^* \sbra{ \m{Z}^l\m{Z}^{l,T} - \cG\m{Y}_{:,l} } }^2 } \right.\\
		& \left. { \qquad + \frac{1}{4}\frobn{(\mathcal{I}-\cG\cG^*)(\m{Z}^l\m{Z}^{l,T})}^2} } \\
		& \qquad + \frac{1}{4}\frobn{(\mathcal{I}-\cW\cW^*)(\m{Z}^1 \m{Z}^{1,H})}^2 \\
		& \qquad + \frac{1}{4}\sum^L_{l=2} \frobn{ \m{Z}^1 \m{Z}^{1,H} - \m{Z}^l \m{Z}^{l,H} }^2.
	\end{split} \label{eq:FactoredModel}
}

As in Algorithm \ref{alg:MHTGD}, we propose a gradient descent algorithm, named CHTGD, to solve the problem in \eqref{eq:FactoredModel}. Let $\m{U}^l\m{\Sigma}^l\m{U}^{l,T} $ be the $K$-truncated Takagi factorization \cite{chebotarev2014singular} of the matrix $p^{-1}\cG\mathcal{P}_{\Omega}(\m{Y}_{:,l})$. Differently from MHTGD, in CHTGD, $\m{Z}^l$ is initialized by $\m{Z}^{l,0} = \m{U}^l\m{\Sigma}^{l,1/2}$ and the gradient $\nabla g \sbra{\lbra{\m{Z}^l}}$ is given by
\equ{ \small{
	\begin{split}
		\nabla g(\m{Z}^1)
		& = \cG\mbra{ \frac{1}{p} \cP_{\Omega}\sbra{\cG^*\sbra{\m{Z}^1\m{Z}^{1,T}}-\m{Y}_{:,1}} - \cG^*\sbra{\m{Z}^1\m{Z}^{1,T}}}\overline{\m{Z}^1}  \\
		& \quad - \cW\cW^*\sbra{\m{Z}^1\m{Z}^{1,H}}\m{Z}^1 + \m{Z}^1\sbra{\m{Z}^{1,T}\overline{\m{Z}^1} +L\m{Z}^{1,H}\m{Z}^1} \\
		& \quad - \sum^L_{q=2} \m{Z}^q \sbra{ \m{Z}^{q,H} \m{Z}^1 }, \nonumber
	\end{split} 
	}
}
\equ{ \small{
	\begin{split}
		& \nabla g(\m{Z}^l)
		= \cG\mbra{ \frac{1}{p} \cP_{\Omega}\sbra{\cG^*\sbra{\m{Z}^l\m{Z}^{l,T}}-\m{Y}_{:,l}} - \cG^*\sbra{\m{Z}^l\m{Z}^{l,T}}}\overline{\m{Z}^l} \\
		& \qquad \; + \m{Z}^l\sbra{\m{Z}^{l,T}\overline{\m{Z}^l} +\m{Z}^{l,H}\m{Z}^l} - \m{Z}^1 \sbra{ \m{Z}^{1,H} \m{Z}^l },  l = 2,\ldots,L. \nonumber
	\end{split} 
	}
}
CHTGD exhibits the same computational complexity as MHTGD. 

\section{Numerical Simulations} \label{sec:sim} 
We present numerical results to illustrate the performance of MHTGD and CHTGD. 
The algorithms stop when the condition $\frobn{\m{X}^{t+1} - \m{X}^t } / \frobn{\m{X}^t} \le 10^{-6}$ is met, or when they reach a maximum of $10^4$ iterations. The step size is determined via a backtracking line search following the Armijo criterion.
We compare our algorithms with ANM \cite{yang2016exact,li2015off} for general multichannel signals and SACA \cite{wu2024multichannel} for CA signals, both solved via an interior point method with CVX \cite{cvx}. 

We study the signal recovery performance in terms of phase transition, considering $N=65$ and $L=5$, with frequencies randomly generated with a minimum separation $1.5/N$. 
A signal is successfully recovered if $\frobn{ \smash{\widehat{\m{X}}} - \m{X}^\star }^2 / \frobn{\m{X}^\star}^2 \le 10^{-6}$.
The success rate is calculated by averaging over $20$ Monte Carlo trials for each combination $\sbra{M,K}$. As shown in Figure \ref{fig:1}, MHTGD and CHTGD perform comparably to ANM and SACA, respectively. 

\begin{figure}[htb] 
	\centering
	\subfigure{\includegraphics[width=3.6cm]{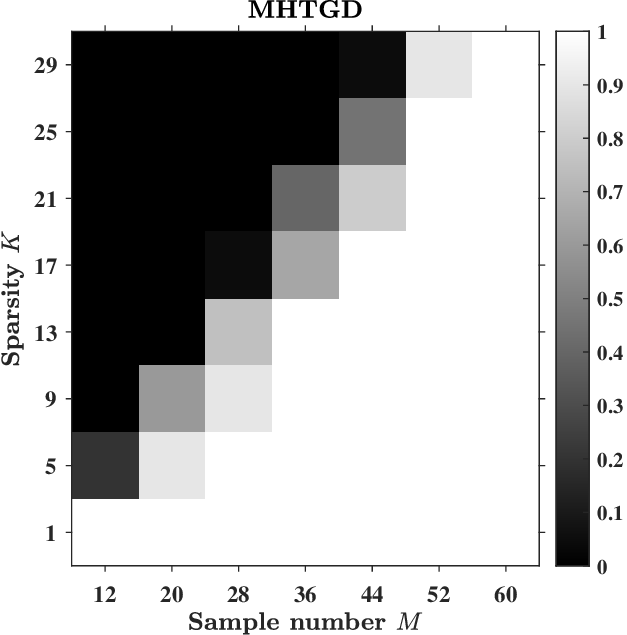}}\hspace{0.3cm}
	\subfigure{\includegraphics[width=3.6cm]{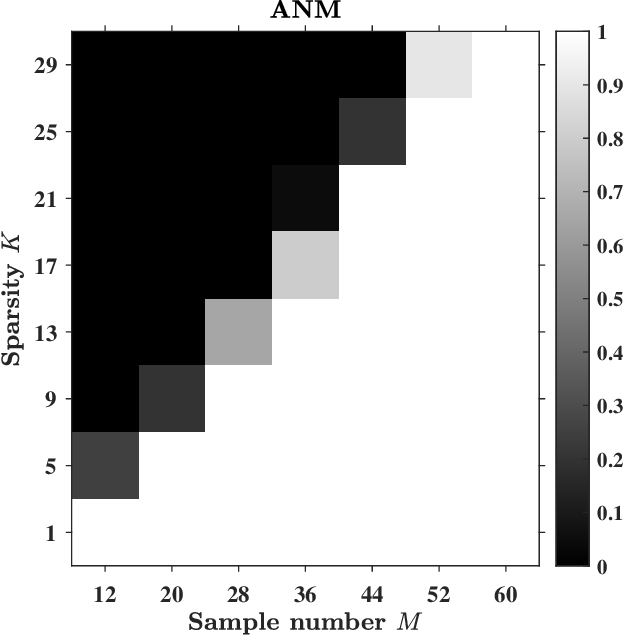}}
	
	\subfigure{\includegraphics[width=3.6cm]{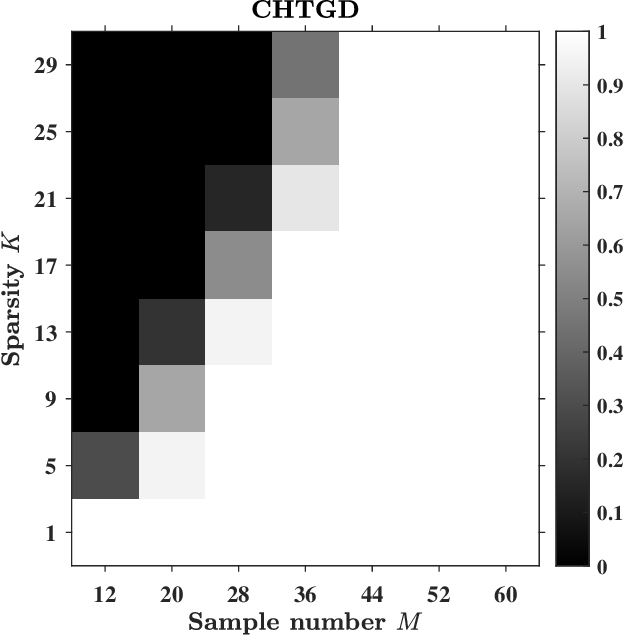}}\hspace{0.3cm}
	\subfigure{\includegraphics[width=3.6cm]{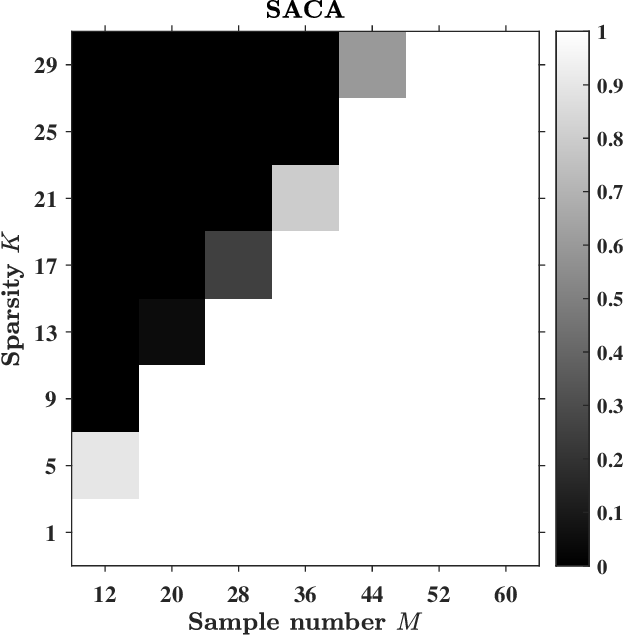}}
	\caption{Phase transitions of MHTGD and ANM for general multichannel signals, and CHTGD and SACA for CA signals. White means success, and vice versa.}
	\label{fig:1}
\end{figure}

We also evaluate the computational efficiency by setting $M = \lfloor 0.8N \rfloor$, $L=3$, and $K=3$ while varying $N$. Methods exceeding $100$ seconds in computational time are excluded from comparison.
The average time over 50 successful signal recovery trials is shown in Figure \ref{fig:time_vs_N}. We can see that MHTGD and CHTGD are significantly faster than ANM and SACA.

\begin{figure}[htbp]
	\centerline{\includegraphics[width=5.8cm]{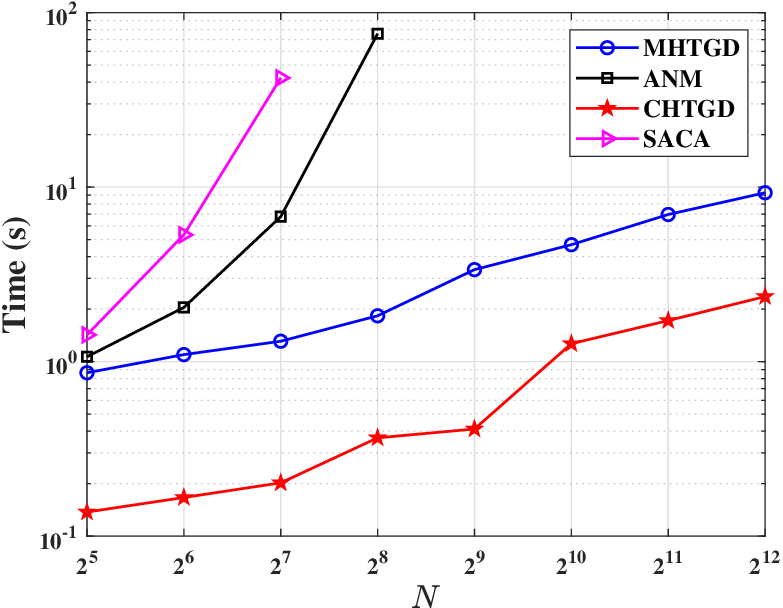}}
	\caption{Time versus $N$.}
	\label{fig:time_vs_N}
\end{figure}

\section{Conclusion}
In this paper, two low-rank Hankel-Toeplitz matrix factorization  problems and the associated low-complexity algorithms were proposed for multichannel spectral super-resolution: one for the case with constant amplitude (CA) and the other for the case without. Simulation results show the superiority of the proposed algorithms.

\newpage

\bibliographystyle{IEEEtran}
\bibliography{MHTGD}

\end{document}